\documentclass{amsproc}%
\usepackage{amsfonts}
\usepackage{amsmath}
\usepackage{amssymb}
\usepackage{graphicx}
\usepackage{hyperref}%
\providecommand{\U}[1]{\protect\rule{.1in}{.1in}}
\theoremstyle{plain}

\newtheorem{corollary}{Corollary}

\newtheorem{definition}{Definition}
\newtheorem{example}{Example}

\newtheorem{lemma}{Lemma}

\newtheorem{proposition}{Proposition}

\newtheorem{theorem}{Theorem}
\numberwithin{equation}{section}
\begin{document}
\title[{\normalsize On S-Comultiplication Modules}]{{\normalsize On S-Comultiplication Modules}}
\author{Eda Y\i ld\i z}
\address{Department of Mathematics, Yildiz Technical University, Istanbul, Turkey.}
\email{edyildiz@yildiz.edu.tr}
\author{\"{U}nsal Tekir}
\address{Department of Mathematics, Marmara University, Istanbul, Turkey.}
\email{utekir@marmara.edu.tr}
\author{Suat Ko\c{c}}
\address{Department of Mathematics, Marmara University, Istanbul, Turkey.}
\email{suat.koc@marmara.edu.tr}
\subjclass[2000]{13C13, 13C99.}
\keywords{S-multiplication module, S-comultiplcation module, S-prime submodule, S-second submodule}

\begin{abstract}
Let $R\ $be a commutative ring with $1\neq0$ and $M$ be an $R$-module. Suppose
that $S\subseteq R\ $is a multiplicatively closed set of $R.\ $Recently, Sevim
et al. in (\cite{SenArTeKo}, Turk. J. Math. (2019)) introduced the notion of
$S$-prime submodule which is a generalization of prime submodule and used them
to characterize certain class of rings/modules such as prime submodules,
simple modules, torsion free modules,\ $S$-Noetherian modules and etc.
Afterwards, in (\cite{AnArTeKo}, Comm. Alg. (2020)), Anderson et al. defined
the concept of $S$-multiplication modules and $S$-cyclic modules which are
$S$-versions of multiplication and cyclic modules and extended many results on
multiplication and cyclic modules to $S$-multiplication and $S$-cyclic
modules. Here, in this article, we introduce and study $S$-comultiplication
module which is the dual notion of $S$-multiplication module. We also
characterize certain class of rings/modules such as comultiplication modules,
$S$-second submodules, $S$-prime ideals, $S$-cyclic modules in terms of
$S$-comultiplication modules.

\end{abstract}
\maketitle

\section{\bigskip Introduction}

Throughout this article, we focus only on commutative rings with a unity and
nonzero unital modules. Let $R\ $will always denote such a ring and $M\ $will
denote such an $R$-module. This paper aims to introduce and study the concept
of $S$-comultiplication module which is both the dual notion of $S$%
-multiplication modules and a generalization of comultiplication modules.
Sevim et al. in their paper \cite{SenArTeKo} gave the concept of $S$-prime
submodules and used them to characterize certain classes of rings/modules such
as prime submodules, simple modules, torsion-free modules and S-Noetherian
rings. A nonempty subset $S$ of $R\ $is said to be a \textit{multiplicatively
closed set} (briefly, \textit{m.c.s}) of $R\ $if $0\notin S,1\in S$ and $st\in
S\ $for each $s,t\in S$. From now on $S\ $will always denote a m.c.s of
$R.\ $Suppose that $P\ $is a submodule of $M$, $K$ is a nonempty subset of
$M\ $and $J$ is an ideal of $R.\ $Then the residuals of $P\ $by $K\ $and
$J\ $are defined as follows:%
\begin{align*}
(P  &  :K)=\{x\in R:xK\subseteq P\}\\
(P  &  :_{M}J)=\{m\in M:Jm\subseteq P\}.
\end{align*}
In particular, if $P=0,\ $we sometimes use $ann(K)\ $instead of $(0:K).\ $%
Recall from \cite{SenArTeKo} that a submodule $P\ $of $M\ $is said to be an
$S$\textit{-prime submodule} if $(P:M)\cap S=\emptyset$ and there exists $s\in
S\ $such that $am\in P\ $for some $a\in R\ $and $m\in M\ $implies either
$sa\in(P:M)$ or $sm\in P.\ \ $Particularly, an ideal $I\ $of $R\ $is said to
be an $S$\textit{-prime ideal} if $I\ $is an $S$-prime submodule of $M$.\ We
here note that if $S\subseteq u(R),$ where $u(R)\ $is the set of all units in
$R,\ $the notion of $S$-prime submodule is in fact prime submodule.

Recall that an $R$-module $M\ $is said to be a \textit{multiplication module}
if each submodule $N\ $of $M\ $has the form $N=IM\ $for some ideal $I\ $of
$R\ $\cite{Bar}. It is easy to note that $M\ $is a multiplication module if
and only if $N=(N:M)M\ $\cite{Smi}. The author in \cite{Smi} showed that for a
multiplication module $M,\ $a submodule $N\ $of $M\ $is prime if and only if
$(N:M)\ $is a prime ideal of $R$ \cite[Corollary 2.11]{Smi}.

The dual notion of prime submodule which is called second submodule was first
introduced and studied by S. Yassemi in \cite{Yas}. Recall from that a nonzero
submodule $P\ $of $M\ $is said to be a second submodule if for each $a\in
R,\ $the homothety $P\overset{a.}{\longrightarrow}P\ $is either zero or
surjective. Note that if $P\ $is a second submodule of $M$, then $ann(P)\ $is
a prime ideal of $R.\ $For the last twenty years, the dual notion of prime
submodule has attracted many researchers and it has been studied in many
papers. See, for example, \cite{AnFa1}, \cite{AnFa2}, \cite{AnFa3},
\cite{AnFa5}, \cite{CeAlsmi} and \cite{CeAl2}. Also the notion of
comultiplication module which is the dual notion of multiplication module was
first introduced by Ansari-Toroghy and Farshadifar in \cite{AnFa4} and has
been widely studied by many authors. See, for instance, \cite{AlSmi}%
,\ \cite{HanKeyFar}, \cite{At} and \cite{Ce}. Recall from \cite{AnFa4} that an
$R$-module $M\ $is said to be a \textit{comultiplication module} if each
submodule $N\ $of $M\ $has the form $N=(0:_{M}I)\ $for some ideal $I\ $of
$R.\ $Note that $M\ $is a comultiplication module if and only if
$N=(0:_{M}ann(N)).\ $

Recently, Anderson et al. in \cite{AnArTeKo}, introduced the notions of
$S$-multiplication modules and $S$-cyclic modules, and they extended many
properties of multiplication and cyclic modules to these two new classes of
modules. They also showed that for $S$-multiplication modules, any submodule
$N\ $of $M\ $is $S$-prime submodule if and only if $(N:M)\ $is an $S$-prime
ideal of $R\ $\cite[Proposition 4]{AnArTeKo}. An $R$-module $M\ $is said to be
an $S$\textit{-multiplication module }if for each submodule $N\ $of
$M,\ $there exist $s\in S\ $and an ideal $I\ $of $R\ $such that $sN\subseteq
IM\subseteq N.\ $Also $M\ $is said to be an $S$-cyclic module if there exists
$s\in S\ $such that $sM\subseteq Rm\ $for some $m\in M.\ $They also showed
that every $S$-cyclic module is an $S$-multiplication module and they
characterized finitely generated multiplication modules in terms of $S$-cyclic
modules\ (See, \cite[Proposition 5]{AnArTeKo} and \cite[Proposition
8]{AnArTeKo}).

Farshadifar, currently, in her paper \cite{Fa} defined the dual notion of
$S$-prime submodule which is called $S$-second submodule and investigate its
many properties similar to second submodules. Recall that a submodule $N\ $of
$M\ $is said to be an $S$-second if $ann(N)\cap S=\emptyset$ and there exists
$s\in S\ $such that either $saN=0\ $or $saN=sN\ $for each $a\in R.\ $In
particular, the author in \cite{Fa} investigate the $S$-second submodules of
comultiplication modules. Here, we introduce $S$-comultiplication module which
is the dual notion of $S$-multiplication modules and investigate its many
properties. Recall that an $R$-module $M\ $is said to be an $S$%
-comultiplication module if for each submodule $N\ $of $M,\ $there exist an
$s\in S\ $and an ideal $I$ of $R\ $such that $s(0:_{M}I)\subseteq
N\subseteq(0:_{M}I).$

Among other results in this paper, we chracterize certain classes of
rings/modules such as comultiplication modules, $S$-second submodules,
$S$-prime ideals, $S$-cyclic modules (See, Theorem \ref{tloc}, Theorem
\ref{tcom}, Proposition \ref{pcy1}, Theorem \ref{torsion}, Theorem \ref{tcy2},
Theorem \ref{tcy3} and Theorem \ref{tm3}). Also, we prove the S-version of
Dual Nakayama's Lemma (See, Theorem \ref{tdu}).

\section{S-comultiplication modules}

\begin{definition}
Let $M\ $be an $R$-module and $S\subseteq R$ be a m.c.s of $R.\ M\ $is said to
be an $S$-comultiplication module if for each submodule $N\ $of $M,\ $there
exist an $s\in S\ $and an ideal $I$ of $R\ $such that $s(0:_{M}I)\subseteq
N\subseteq(0:_{M}I).\ $In particular, a ring $R\ $is said to be an
$S$-comultiplication ring if it is an $S$-comultiplication module over itself.
\end{definition}

\begin{example}
Every $R$-module $M\ $with $ann(M)\cap S\neq\emptyset\ $is trivially an
$S$-comultiplication module.
\end{example}

\begin{example}
\textbf{(An S-comultiplication module that is not S-multiplication)} Let
$p\ $be a prime number and consider the $%
\mathbb{Z}
$-module
\[
E(p)=\{\alpha=\frac{m}{p^{n}}+%
\mathbb{Z}
:m\in%
\mathbb{Z}
,n\in%
\mathbb{N}
\cup\{0\}\}.
\]
Then every submodule of $E(p)$ is of the form $G_{t}=\{\alpha=\frac{m}{p^{t}}+%
\mathbb{Z}
:m\in%
\mathbb{Z}
\}$ for some fixed $t\geq0.$\ Take the multiplicatively closed set
$S=\{1\}.\ $Note that $(G_{t}:E(p))E(p)=0_{E(p)}\neq G_{t}\ $for each
$t\geq1.\ $Then $E(p)\ $is not an $S$-multiplication module. Now, we will show
that $E(p)\ $is an $S$-comultiplication module. Let $t\geq0.\ $Then it is easy
to see that $(0:_{E(p)}ann(G_{t}))=(0:_{E(p)}p^{t}%
\mathbb{Z}
)=G_{t}.\ $Therefore, $E(p)\ $is an $S$-comultiplication module.
\end{example}

\begin{example}
\label{exco}Every comultiplication module is also an $S$-comultiplication
module. Also the converse is true provided that $S\subseteq u(R).$
\end{example}

\begin{example}
\textbf{(An S-comultiplication module that is not comultiplication)} Consider
the $%
\mathbb{Z}
$-module $M=%
\mathbb{Z}
$ and $S=reg(%
\mathbb{Z}
)=%
\mathbb{Z}
-\{0\}.\ $Now, take the submodule $N=m%
\mathbb{Z}
$,\ where $m\neq0,\pm1.\ $Then $(0:ann(m%
\mathbb{Z}
))=%
\mathbb{Z}
\neq m%
\mathbb{Z}
$ so that $M$ is not a comultiplication module. Now, take a submodule
$K\ $of$\ M.\ $Then $K=k%
\mathbb{Z}
$ for some $k\in%
\mathbb{Z}
.\ $If $k=0,\ $then choose $s=1\ $and note that $s(0:ann(K))=(0)=k%
\mathbb{Z}
.\ $If $k\neq0,\ $then choose $s=k\ $and note that $s(0:ann(K))\subseteq k%
\mathbb{Z}
=K\subseteq(0:ann(K)).\ $Therefore, $M\ $is an $S$-comultiplication module.
\end{example}

\begin{lemma}
\label{lma}Let $M\ $be an $R$-module. The following statements are equivalent.

(i) $M\ $is an $S$-comultiplication module.

(ii)\ For each submodule $N\ $of $M,\ $there exists $s\in S\ $such that
$s(0:_{M}ann(N))\subseteq N\subseteq(0:_{M}ann(N)).\ $

(iii)\ For each submodule $K,N\ $of $M\ $with $ann(K)\subseteq ann(N),\ $there
exists $s\in S\ $such that $sN\subseteq K.$
\end{lemma}

\begin{proof}
$(i)\Rightarrow(ii):\ $Suppose that $M\ $is an $S$-comultiplication module and
take a submodule $N$ of $M.\ $Then by definition, there exist $s\in S\ $and an
ideal $I\ $of $R\ $such that $s(0:_{M}I)\subseteq N\subseteq(0:_{M}I).\ $Then
note that $IN=(0)\ $and so $I\subseteq ann(N).\ $This gives that
$s(0:_{M}ann(N))\subseteq s(0:_{M}I)\subseteq N\subseteq(0:_{M}ann(N))\ $which
completes the proof.

$(ii)\Rightarrow(iii):\ $Suppose that $ann(K)\subseteq ann(N)$ for some
submodules $N,K\ $of $M.\ $By (ii), there exist $s_{1},s_{2}\in S$ such that
\begin{align*}
s_{1}(0  &  :_{M}ann(N))\subseteq N\subseteq(0:_{M}ann(N))\\
s_{2}(0  &  :_{M}ann(K))\subseteq K\subseteq(0:_{M}ann(K)).
\end{align*}
Since $ann(K)\subseteq ann(N),\ $we have $(0:_{M}ann(N))\subseteq
(0:_{M}ann(K))\ $and so
\begin{align*}
s_{1}s_{2}(0  &  :_{M}ann(N))\subseteq s_{2}N\subseteq s_{2}(0:_{M}ann(N))\\
&  \subseteq s_{2}(0:_{M}ann(K))\subseteq K
\end{align*}
which completes the proof.

$(iii)\Rightarrow(ii):\ $Suppose that (iii) holds. Let $N\ $be a submodule of
$M.\ $Then it is clear that $ann(N)=ann(0:_{M}ann(N)).\ $Then by (iii), there
exists $s\in S\ $such that $s(0:_{M}ann(N))\subseteq N\subseteq(0:_{M}%
ann(N)).$

$(ii)\Rightarrow(i):\ $It is clear.
\end{proof}

Let $S\ $be a m.c.s of $R.\ $The \textit{saturation} $S^{\star}\ $of $S\ $is
defined by $S^{\star}=\{x\in R:x|s\ $for some $s\in S\}.\ $Also $S\ $is said
to be a \textit{saturated m.c.s} of $R$ if $S=S^{\star}.\ $Note that
$S^{\star}$ is always a satured m.c.s of $R\ $containing $S.$

\begin{proposition}
Let $M$ be an $R$-module and $S$ be a m.c.s of $R$. The following assertions hold.

(i) Let $S_{1}$ and $S_{2}$ be two m.c.s of $R$ and $S_{1}\subseteq S_{2}$. If
$M$ is an $S_{1}$-comultiplication module, then $M$ is also an $S_{2}%
$-comultiplication module.

(ii) $M$ is an $S$-comultiplication module if and only if $M$ is an $S^{\star
}$-comultiplication module, where $S^{\ast}$ is the saturation of $S$.
\end{proposition}

\begin{proof}
(i): Clear.

(ii): Assume that $M$ is an $S$-comultiplication module. Since $S\subseteq
S^{\star}$, the result follows from the part (i).

Suppose $M$ is an $S^{\star}$-comultiplication module. Take a submodule $N$ of
$M$. Since $M$ is $S^{\star}$-comultiplication module, there exists $x\in
S^{\star}$ such that $x(0:_{M}ann(N))\subseteq N\subseteq(0:_{M}ann(N))$ by
Lemma \ref{lma}. Since $x\in S^{\star}$, there exists $s\in S$ such that
$x|s$, that is, $s=rx$ for some $r\in R$. This implies that $s(0:_{M}%
ann(N))\subseteq x(0:_{M}ann(N))\subseteq N\subseteq(0:_{M}ann(N))$. Thus, $M$
is an $S$-comultiplication module.
\end{proof}

Anderson and Dumitrescu, in 2002, defined the concept of $S$-Noetherian rings
which is a generalization of Noetherian rings and they extended many
properties of Noetherian rings to $S$-Noetherian rings. Recall from
\cite{AnDu} that a submodule $N\ $of $M\ $is said to be an $S$\textit{-finite
submodule} if there exists a finitely generated submodule $K\ $of $M\ $such
that $sN\subseteq K\subseteq N.\ $Also, $M\ $is said to be an $S$%
\textit{-Noetherian module} if its each submodule is $S$-finite. In
particular, $R\ $is said to be an $S$-Neotherian ring if it is an
$S$-Noetherian $R$-module.

\begin{proposition}
Let $R$ be an $S$-Noetherian ring and $M$ be an $S$-comultiplication module.
Then $S^{-1}M$ is a comultiplication module.
\end{proposition}

\begin{proof}
Let $W$ be a submodule of $S^{-1}M$. Then, $W=S^{-1}N$ for some submodule $N$
of $M$. Since $M$ is an $S$-comultiplication module, there exists $s\in S$
such that $s(0:_{M}I)\subseteq N\subseteq(0:_{M}I)$ for some ideal $I$ of $R$.
Then, we get $S^{-1}(s(0:_{M}I))=S^{-1}((0:_{M}I))\subseteq S^{-1}N\subseteq
S^{-1}((0:_{M}I))$ that is $S^{-1}N=S^{-1}((0:_{M}I))$. Now, we will show that
$S^{-1}((0:_{M}I))=(0:_{S^{-1}M}S^{-1}I)$. Let $\frac{m}{s^{\prime}}\in
S^{-1}((0:_{M}I))$ where $m\in(0:_{M}I)$ and $s^{\prime}\in S$. Then, we have
$Im=(0)$ and so $(S^{-1}I)(\frac{m}{s^{\prime}})=(0)$. This implies that
$\frac{m}{s^{\prime}}\in(0:_{S^{-1}M}S^{-1}I)$. For the converse, let
$\frac{m}{s^{\prime}}\in(0:_{S^{-1}M}S^{-1}I)$. Then, we have $(S^{-1}%
I)(\frac{m}{s^{\prime}})=(0)$. This implies that, for each $x\in I$, there
exists $s^{\prime\prime}\in S$ such that $s^{\prime\prime}xm=0$. Since $R$ is
an $S$-Noetherian ring, $I$ is $S$-finite. So, there exists $s^{\star}\in S$
and $a_{1},a_{2},\ldots,a_{n}\in I$ such that $s^{\star}I\subseteq(a_{1}%
,a_{2},\ldots,a_{n})\subseteq I$. As $\left(  S^{-1}I\right)  (\frac
{m}{s^{\prime}})=(0)$ and $a_{i}\in I$, there exists $s_{i}\in S$ such that
$s_{i}a_{i}m=0$. Now, put $t=s_{1}s_{2}\cdots s_{n}s^{\star}\in S$. Then we
have $ta_{i}m=0$ for all $a_{i}$ and so $tIm=0$. Then we deduce $\frac
{m}{s^{\prime}}=\frac{tm}{ts^{\prime}}\in S^{-1}((0:_{M}I))$. Thus,
$S^{-1}((0:_{M}I))=(0:_{S^{-1}M}S^{-1}I)$ and so $W=S^{-1}N=(0:_{S^{-1}%
M}S^{-1}I)$. Therefore, $S^{-1}M$ is a comultiplication module.
\end{proof}

Recall from \cite{AnArTeKo} that a m.c.s $S$ of $R$ is said to satisfy
\textit{maximal multiple condition} if there exists $s\in S$ such that $t$
divides $s$ for each $t\in S$.

\begin{theorem}
\label{tloc}Let $M$ be an $R$-module and $S$ be a m.c.s. of $R$ satisfying
maximal multiple condition. Then, $M$ is an $S$-comultiplication module if and
only if $S^{-1}M$ is a comultiplication module.
\end{theorem}

\begin{proof}
$(\Rightarrow):\ $Suppose that $W\ $is a submodule of $S^{-1}M.\ $Then
$W=S^{-1}N\ $for some submodule $N\ $of $M.\ $Since $M\ $is an $S$%
-comultiplication module, there exist $t^{\prime}\in S\ $and an ideal $I$ of
$R\ $such that $t^{\prime}(0:_{M}I)\subseteq N\subseteq(0:_{M}I).\ $This
implies that $IN=(0)\ $and so $S^{-1}(IN)=(S^{-1}I)(S^{-1}N)=0.\ $Then we have
$S^{-1}N\subseteq(0:_{S^{-1}M}S^{-1}I).\ $Let $\frac{m^{\prime}}{s^{\prime}%
}\in(0:_{S^{-1}M}S^{-1}I).\ $Then we get $\frac{a}{1}\frac{m^{\prime}%
}{s^{\prime}}=0$ for each $a\in I\ $and this yields that $uam^{\prime}=0$ for
some $u\in S.\ $As $S\ $satisfies maximal multiple condition, there exists
$s\in S\ $such that $u|s\ $for each $u\in S.\ $This implies that $s=ux$ for
some $x\in R.\ $Then we have $sam^{\prime}=xuam^{\prime}=0.\ $Then we conclude
that $Ism^{\prime}=0$ and so $sm^{\prime}\in(0:_{M}I).\ $This yields that
$t^{\prime}sm^{\prime}\in t^{\prime}(0:_{M}I)\subseteq N\ $and so
$\frac{m^{\prime}}{s^{\prime}}=\frac{t^{\prime}sm^{\prime}}{t^{\prime
}ss^{\prime}}\in S^{-1}N.\ $Then we get $S^{-1}N=(0:_{S^{-1}M}S^{-1}I)$ and so
$S^{-1}M$ is a comultiplication module.

$(\Leftarrow):\ $Suppose that $S^{-1}M$ is a comultiplication module. Let
$N\ $be a submodule of $M.\ $Since $S^{-1}M$ is comultiplication,
$S^{-1}N=(0:_{S^{-1}M}S^{-1}I)\ $for some ideal $I\ $of $R.\ $Then we have
$(S^{-1}I)(S^{-1}N)=S^{-1}(IN)=0.\ $Then for each $a\in I,m\in N,\ $we have
$\frac{am}{1}=0\ $and thus $uam=0\ $for some $u\in S.\ $By maximal multiple
condition, there exists $s\in S\ $such that $sam=0\ $and so $sIN=0.\ $This
implies that $N\subseteq(0:_{M}sI).\ $Now, let $m\in(0:_{M}sI).\ $Then
$Ism=0\ $so it is easily seen that $(S^{-1}I)\frac{m}{1}=0.\ $Then we conclude
that $\frac{m}{1}\in(0:_{S^{-1}M}S^{-1}I)=S^{-1}N.\ $Then there exists $x\in
S\ $such that $xm\in N.\ $Again by maximal multiple condition, $sm\in N.$ Then
we have $s(0:_{M}sI)\subseteq N\subseteq(0:_{M}sI).\ $Since $sI\ $is an ideal
of $R,$ $M\ $is an $S$-comultiplication module.
\end{proof}

\begin{theorem}
Let $f:M\rightarrow M^{\prime}$ be an $R$-homomorphism and $tKer(f)=(0)$ for
some $t\in S$.

(i)\ If $M^{\prime}$ is an $S$-comultiplication module, then $M$ is an
$S$-comultiplication module.

(ii)\ If $f$ is an $R$-epimorphism and $M$ is an $S$-comultiplication module,
then $M^{\prime}$ is an $S$-comultiplication module.
\end{theorem}

\begin{proof}
(i)\ Let $N$ be a submodule of $M$. Since $M^{\prime}$ is an $S$%
-comultiplication module, there exist $s\in S$ and an ideal $I$ of $R$ such
that $s(0:_{M^{\prime}}I)\subseteq f(N)\subseteq(0:_{M^{\prime}}I)$. Then, we
have $If(N)=f(IN)=0$ and so $IN\subseteq Kerf$. Since $tKer(f)=0$, we have
$tIN=(0)$ and so $N\subseteq(0:_{M}tI)$. Now, we will show that $t^{2}%
s(0:_{M}tI)\subseteq N\subseteq(0:_{M}tI)$. Let $m\in(0:_{M}tI)$. Then, we
have $tIm=0$ and so $f(tIm)=tIf(m)=If(tm)=0$. This implies that $f(tm)\in
(0:_{M^{\prime}}I)$. Thus, we have $sf(tm)=f(stm)\in s(0:_{M^{\prime}%
}I)\subseteq f(N)$ and so there exists $y\in N$ such that $f(stm)=f(y)$ and so
$stm-y\in Ker(f)$. Thus, we have $t(stm-y)=0$ and so $t^{2}sm=tx$. Then we
obtain
\[
t^{2}s(0:_{M}tI)\subseteq tN\subseteq N\subseteq(0:_{M}tI).
\]
Now, put $t^{2}s=s^{\prime}\in S$ and $J=tI$. Thus,
\[
s^{\prime}(0:_{M}J)\subseteq N\subseteq(0:_{M}J).
\]
Therefore, $M$ is an $S$-comultiplication module.

(ii)\ Let $N^{\prime}$ be a submodule of $M^{\prime}$. Since $M$ is an
$S$-comultiplication module, there exist $s\in S$ and an ideal $I$ of $R$ such
that
\[
s(0:_{M}I)\subseteq f^{-1}(N^{\prime})\subseteq(0:_{M}I).
\]
This implies that $If^{-1}(N^{\prime})=(0)$ and so $f(If^{-1}(N^{\prime
}))=IN^{\prime}=(0)$ since $f$ is surjective. Then, we have $N^{\prime
}\subseteq(0:_{M^{\prime}}I)$. On the other hand, we get $f(s(0:_{M}%
I))=sf((0:_{M}I))\subseteq f(f^{-1}(N^{\prime}))=N^{\prime}$. Now, let
$m^{\prime}\in(0:_{M^{\prime}}I)$. Then, $Im^{\prime}=0$. Since, $f$ is
epimorphism, there exists $m\in M$ such that $m^{\prime}=f(m)$. Then, we have
$Im^{\prime}=If(m)=f(Im)=0$ and so $Im\subseteq Kerf$. Since $tKer(f)=0$, we
have $tIm=(0)$ and so $tm\in(0:_{M}I)$. Then we get $f(tm)=tf(m)=tm^{\prime
}\in f((0:_{M}I))$. Thus, we have $t(0:_{M^{\prime}}I)\subseteq f((0:_{M}I))$
and hence $st(0:_{M^{\prime}}I)\subseteq sf((0:_{M}I))\subseteq N^{\prime
}\subseteq(0:_{M^{\prime}}I)$. Thus, $M^{\prime}$ is an $S$-comultiplication module.
\end{proof}

As an immediate consequences of previous theorem, we give the following
explicit results.

\begin{corollary}
Let $M\ $be an $R$-module,$\ N$ be a submodule of $M\ $and $S$ be a m.c.s of
$R$. Then we have the following.

(i) If $M\ $is an $S$-comultiplication module, then $N\ $is an $S$%
-comultiplication module.

(ii)\ If $M\ $is an $S$-comultiplication module and $tM\subseteq N\ $for some
$t\in S,\ $then $M/N\ $is an $S$-comultiplication $R$-module.
\end{corollary}

\begin{proposition}
\label{pcar}Let $M_{i}\ $be an $R_{i}$-module and $S_{i}$ be a m.c.s of
$R_{i}\ $for each $i=1,2.\ $Suppose that $M=M_{1}\times M_{2},\ R=R_{1}\times
R_{2}\ $and $S=S_{1}\times S_{2}.$\ The following assertions are equivalent.

(i) $M$ is an $S$-comultiplication $R$-module.

(ii) $M_{1}$ is an $S_{1}$-comultiplication $R_{1}$-module and $M_{2}$ is an
$S_{2}$-comultiplication $R_{2}$-module.
\end{proposition}

\begin{proof}
$(i)\Rightarrow(ii):\ $Assume that $M$ is an $S$-comultiplication $R$-module.
Take a submodule $N_{1}\ $of $M_{1}$. Then, $N_{1}\times\{0\}$ is a submodule
of $M$. Since $M$ is an $S$-comultiplication module, there exist
$s=(s_{1},s_{2})\in S_{1}\times S_{2}$ and an ideal $J=I_{1}\times I_{2}\ $of
$R$ such that $(s_{1},s_{2})(0:_{M}I_{1}\times I_{2})\subseteq N_{1}%
\times\{0\}\subseteq(0:_{M}I_{1}\times I_{2})$,$\ $where $I_{i}\ $is an ideal
of $R_{i}$. Then we can easily get $s_{1}(0:_{M_{1}}I_{1})\subseteq
N_{1}\subseteq(0:_{M_{1}}I_{1})$ which shows that $M_{1}$ is an $S_{1}%
$-comultiplication module. Similarly, taking a submodule $N_{2}\ $of $M_{2}$
and a submodule $\{0\}\times N_{2}$ of $M$, we can show that $M_{2}$ is an
$S_{2}$-comultiplication module.

$(ii)\Rightarrow(i):\ $Now, assume that $M_{1}$ is an $S_{1}$-comultiplication
module and $M_{2}$ is an $S_{2}$-comultiplication module. Let $N\ $be a
submodule of $M.\ $Then we can write $N=N_{1}\times N_{2}\ $for some submodule
$N_{i}\ $of $M_{i}$.\ Since $M_{1}$ is an $S_{1}$-comultiplication module,
\[
s_{1}(0:_{M_{1}}I_{1})\subseteq N_{1}\subseteq(0:_{M_{1}}I_{1})
\]
for some ideal $I_{1}$ of $R_{1}$ and $s_{1}\in S_{1}$. Since $M_{2}$ is an
$S_{2}$-comultiplication module,
\[
s_{2}(0:_{M_{2}}I_{2})\subseteq N_{2}\subseteq(0:_{M_{2}}I_{2})
\]
for some ideal $I_{2}$ of $R_{2}$ and $s_{2}\in S_{2}$. Put $s=(s_{1}%
,s_{2})\in S$. Then,
\begin{align*}
s(0  &  :_{M}I_{1}\times I_{2})=s_{1}(0:_{M_{1}}I_{1})\times s_{2}(0:_{M_{2}%
}I_{2})\\
&  \subseteq N_{1}\times N_{2}\subseteq(0:_{M_{1}}I_{1})\times(0:_{M_{2}}%
I_{2})=(0:_{M}I_{1}\times I_{2})
\end{align*}
where $I_{1}\times I_{2}$ is an ideal of $R$ and $(s_{1},s_{2})\in S$, as needed.
\end{proof}

\begin{theorem}
Let $M=M_{1}\times M_{2}\times\cdots\times M_{n}$ be an $R=R_{1}\times
R_{2}\times\cdots\times R_{n}$ module and $S=S_{1}\times S_{2}\times
\cdots\times S_{n}$ be a m.c.s. of $R$ where $M_{i}$ are $R_{i}$-modules and
$S_{i}$ are m.c.s of $R_{i}$ for all $i\in\{{1,2,...,n\}}$, respectively. The
following statements are equivalent.

(i) $M\ $is an $S$-comultiplication $R$-module.

(ii) $M_{i}$ is an $S_{i}$-comultiplication $R_{i}$-module for each
$i=1,2,\ldots,n$.
\end{theorem}

\begin{proof}
Here, induction can be applied on $n$. The statement is true when $n=1$. If
$n=2$, result follows from Proposition \ref{pcar}. Assume that statements are
equivalent for each $k<n$. We will show that it also holds for $k=n$. Now, put
$M^{\prime}=M_{1}\times M_{2}\times\cdots\times M_{n-1},\ R=R_{1}\times
R_{2}\times\cdots\times R_{n-1}\ $and $S=S_{1}\times S_{2}\times\cdots\times
S_{n-1}.\ $Note that $M=M^{\prime}\times M_{n},\ R=R^{\prime}\times R_{n}%
\ $and $S=S^{\prime}\times S_{n}.\ $Then by Proposition \ref{pcar},\ $M\ $is
an $S$-comultiplication $R$-module if and only if $M^{\prime}\ $is an
$S^{\prime}$-comultiplication $R^{\prime}$-module and $M_{n}\ $is an $S_{n}%
$-comultiplication $R_{n}$-module. The rest follows from induction hypothesis.
\end{proof}

Let $p$ be a prime ideal of $R.\ $Then we know that $S_{p}=(R-p)\ $is a m.c.s
of $R.\ $If an $R$-module $M\ $is an $S_{p}$-comultiplication for a prime
ideal $p\ $of $R,\ $then we say that $M\ $is a $p$-comultiplication module.
Now, we will characterize comultiplication modules in terms of $S$%
-comultiplication modules.

\begin{theorem}
\label{tcom}Let $M\ $be an $R$-module. The following statements are equivalent.

(i)\ $M\ $is a comultiplication module.

(ii)\ $M\ $is a $\mathcal{P}$-comultiplication module for each prime ideal
$\mathcal{P}\ $of $R.\ $

(iii)\ $M\ $is an $\mathcal{M}$-comultiplication module for each maximal ideal
$\mathcal{M}$ of $R.$

(iv)\ $M\ $is an $\mathcal{M}$-comultiplication module for each maximal ideal
$\mathcal{M}$ of $R\ $with $M_{\mathcal{M}}\neq0_{\mathcal{M}}.$
\end{theorem}

\begin{proof}
$(i)\Rightarrow(ii):\ $Follows from Example \ref{exco}.

$(ii)\Rightarrow(iii):\ $Follows from the fact that every maximal ideal is prime.

$(iii)\Rightarrow(iv):\ $Clear.

$(iv)\Rightarrow(i):\ $Suppose that $M\ $is an $\mathcal{M}$-comultiplication
module for each maximal ideal $\mathcal{M}$ of $R$ with $M_{\mathcal{M}}%
\neq0_{\mathcal{M}}.\ $Take a submodule $N\ $of $M$ and a maximal ideal
$\mathcal{M}$ of $R.\ $If $M_{\mathcal{M}}=0_{\mathcal{M}},$\ then clearly we
have $N_{\mathcal{M}}=(0:_{M}ann(N))_{\mathcal{M}}.\ $So assume that
$M_{\mathcal{M}}\neq0_{\mathcal{M}}.\ $Since $M\ $is an $\mathcal{M}%
$-comultiplication module, there exists $s_{\mathcal{M}}\notin\mathcal{M\ }%
$such that $s_{\mathcal{M}}(0:_{M}ann(N))\subseteq N.\ $Then we have
\[
(0:_{M}ann(N))_{\mathcal{M}}=\left(  s_{\mathcal{M}}(0:_{M}ann(N))\right)
_{\mathcal{M}}\subseteq N_{\mathcal{M}}\subseteq(0:_{M}ann(N))_{\mathcal{M}}.
\]
Thus we have $N_{\mathcal{M}}=(0:_{M}ann(N))_{\mathcal{M}}\ $for each maximal
ideal $\mathcal{M\ }$of $R.\ $Therefore, $N=(0:_{M}ann(N))$ so that $M\ $is a
comultiplication module.
\end{proof}

Now, we shall give the S-version of Dual Nakayama's Lemma for $S$%
-comultiplication module. First, we need the following Proposition.

\begin{proposition}
\label{pf}Let $M\ $be an $S$-comultiplication $R$-module. Then,

(i) If $I\ $is an ideal of $R\ $with $(0:_{M}I)=0,\ $then there exists $s\in
S\ $such that $sM\subseteq IM.$

(ii)\ If $I\ $is an ideal of $R\ $with $(0:_{M}I)=0,\ $then for every element
$m\in M,\ $there exists $s\in S\ $and $a\in I\ $such that $sm=am.\ $

(iii)\ If $M\ $is an $S$-finite $R$-module and $I\ $is an ideal of $R\ $with
$(0:_{M}I)=0,\ $then there exist $s\in S\ $and $a\in I\ $such that $(s+a)M=0.$
\end{proposition}

\begin{proof}
$(i):$\ Suppose that $I\ $is an ideal of $R\ $with $(0:_{M}I)=0.\ $Then we
have $((0:_{M}I):M)=(0:IM)=(0:M).\ $Then by Lemma \ref{lma} (iii), there
exists $s\in S\ $such that $sM\subseteq IM.\ $

$(ii):$ Suppose that $I\ $is an ideal of $R\ $with $(0:_{M}I)=0.\ $Then for
any $m\in M,\ $we have $(0:Rm)=((0:_{M}I):Rm)=(0:Im).\ $Again by Lemma
\ref{lma} (iii), there exists $s\in S\ $such that $sRm\subseteq Im\ $and so
$sm=am\ $for some $a\in I.$

$(iii):$\ Suppose that $M\ $is an $S$-finite $R$-module and $I\ $is an ideal
of $R\ $with $(0:_{M}I)=0.\ $Then there exists $t\in S\ $such that
$tM\subseteq Rm_{1}+Rm_{2}+\cdots+Rm_{n}\ $for some $m_{1},m_{2},\ldots
,m_{n}\in M.\ $Since $(0:_{M}I)=0,\ $by (i), there exists $s\in S\ $such that
$sM\subseteq IM.$ This implies that $stM\subseteq tIM=ItM\subseteq
I(Rm_{1}+Rm_{2}+\cdots+Rm_{n})=Im_{1}+Im_{2}+\cdots+Im_{n}.\ $Then for each
$i=1,2,\ldots,n,\ $we have $stm_{i}=a_{i1}m_{1}+a_{i2}m_{2}+\cdots+a_{in}%
m_{n}\ $and so $-a_{i1}m_{1}-a_{i2}m_{2}-\cdots+(st-a_{ii})m_{i}+\cdots
-a_{in}m_{n}=0.\ $Now, let $\Delta\ $be the following matrix
\[
\left[
\begin{array}
[c]{cccc}%
st-a_{11} & -a_{12} & \cdots & -a_{1n}\\
-a_{21} & st-a_{22} & \cdots & -a_{2n}\\
\vdots & \vdots & \ddots & \vdots\\
-a_{n1} & -a_{n2} & \cdots & st-a_{nn}%
\end{array}
\right]  _{n\times n}.
\]
Then we have $\left\vert \Delta\right\vert m_{i}=0\ $for each $i=1,2,\ldots
,n.\ $Thus we obtain that $t\left\vert \Delta\right\vert M=0.\ $This implies
that $t(s^{n}t^{n}+a)M=(s^{n}t^{n+1}+at)M=0\ $for some $a\in I.\ $Now, put
$u=s^{n}t^{n+1}\in S$ and $b=at\in I.\ $Then we have $(u+b)M=0\ $which
completes the proof.
\end{proof}

\begin{theorem}
\label{tdu}\textbf{(S-Dual Nakayama's Lemma)} Let $M$ be an $S$%
-comultiplication module, where $S\ $is a m.c.s of $R\ $satisfying maximal
multiple condition. Suppose that $I\ $is an ideal of $R\ $such that
$tI\subseteq Jac(R)\ $for some $t\in S.\ $If $(0:_{M}tI)=0,\ $then there
exists $s\in S\ $such that $sM=0.$
\end{theorem}

\begin{proof}
Suppose that $S\ $satisfies maximal multiple condition. Then there exists
$s\in S\ $such that $t|s$ for each $t\in S.\ $Let $I\ $be an ideal of
$R\ $with $tI\subseteq Jac(R)\ $for some $t\in S\ $and $(0:_{M}tI)=0.\ $Then
for each $m\in M,\ $by Proposition \ref{pf} (ii), there exists $t^{\prime}\in
S\ $such that $t^{\prime}Rm\subseteq tIm\ $and so $s^{2}t^{\prime}Rm\subseteq
s^{2}tIm\subseteq s^{2}Im.\ $Now, put $u=s^{2}t^{\prime}.\ $By maximal
multiple condition, we have $sRm\subseteq uRm\subseteq s^{2}Im\ $and so
$sm=s^{2}am\ $for some $a\in R.\ $On the other hand, we note that $sI\subseteq
tI\subseteq Jac(R).\ $Thus we have\ $s(1-sa)m=0.\ $Since $sa\in Jac(R),\ $we
get $1-sa$ is unit and so $sm=0.\ $Thus we have $sM=0.\ $
\end{proof}

\begin{corollary}
\textbf{(Dual Nakayama's Lemma)} Let $M$ be a comultiplication module and
$I\ $an ideal of $R\ $such that $I\subseteq Jac(R).\ $If $(0:_{M}I)=0,\ $then
$M=0.\ $
\end{corollary}

\begin{proof}
Take $S=\{1\}\ $and apply Theorem \ref{tdu}.
\end{proof}

\section{S-cyclic modules}

In this section, we investigate the relations between $S$-comultiplication
modules and $S$-cyclic modules.

\begin{proposition}
\label{pcy1}Let $M$ be an $S$-comultiplication $R$-module and $N$ be a minimal
ideal of $R$ such that $(0:_{M}N)=0$. Then, $M$ is an $S$-cyclic module.
\end{proposition}

\begin{proof}
Choose a nonzero element $m$ of $M$. Since $M$ is\ an $S$-comultiplication
module, there exist $s\in S$ and an ideal $I$ of $R$ such that $s(0:_{M}%
I)\subseteq Rm\subseteq(0:_{M}I)$. By the assumption $(0:_{M}N)=0$, we have
\[
s((0:_{M}N):_{M}I)\subseteq Rm\subseteq((0:_{M}N):_{M}I)\Longrightarrow
s(0:_{M}NI)\subseteq Rm\subseteq(0:_{M}NI).
\]
Since $0\subseteq NI\subseteq N$ and $N$ is minimal ideal of $R$, either
$NI=N$ or $NI=0$. If the former case holds, we have $s(0:_{M}N)\subseteq
Rm\subseteq(0:_{M}N)$. This means that$\ Rm=0$, a contradiction. The second
case implies the equality $s(0:_{M}0)\subseteq Rm\subseteq(0:_{M}0)$. It means
$sM\subseteq Rm\subseteq M$ proving that $M$ is $S$-cyclic.
\end{proof}

\begin{proposition}
Let $M$ be an $S$-comultiplication module of $R$. Let $\{M_{i}\}$ be a
collection of submodules of $M$ with $\bigcap_{i}M_{i}=0$. Then, for every
submodule $N$ of $M$, there exists an $s\in S$ such that
\[
s\bigcap_{i}(N+M_{i})\subseteq N\subseteq\bigcap_{i}(N+M_{i}).
\]

\end{proposition}

\begin{proof}
Let $N$ be a submodule of $M$. Since $M$ is an $S$-comultiplication module, we
have $s(0:_{M}ann(N))\subseteq N\subseteq(0:_{M}ann(N))$ for some $s\in S$.
This implies $s(\bigcap_{i}M_{i}:_{M}ann(N))\subseteq N\subseteq(\bigcap
_{i}M_{i}:_{M}ann(N))$ since $\bigcap_{i}M_{i}=0$. Then, we obtain
$s\bigcap_{i}(M_{i}:_{M}ann(N))\subseteq N\subseteq\bigcap_{i}(M_{i}%
:_{M}ann(N))$. Thus,
\[
s\bigcap_{i}(N+M_{i})\subseteq s\bigcap_{i}(M_{i}:_{M}ann(N))\subseteq
N\subseteq\bigcap_{i}(N+M_{i}).
\]

\end{proof}

\begin{proposition}
Let $M$ be an $S$-comultiplication module. Then, for each submodule $N$ of $M$
and each ideal $I$ of $R$ with $N\subseteq s(0:_{M}I)$ for some $s\in S$,
there exists an ideal $J$ of $R$ such that $I\subseteq J$ and $s(0:_{M}%
J)\subseteq N$.
\end{proposition}

\begin{proof}
Let $N$ be a submodule of $M$. Since $M$ is an $S$-comultiplication module,
$s(0:_{M}ann(N))\subseteq N\subseteq(0:_{M}ann(N))$ for some $s\in S$. So, we
obtain $s(0:_{M}ann(N))\subseteq N\subseteq s(0:_{M}I)$. Taking $J=I+ann(N)$,
\[
s(0:_{M}J)=s(0:_{M}I+ann(N))\subseteq s(0:_{M}I)\cap s(0:_{M}ann(N))\subseteq
s(0:_{M}ann(N))\subseteq N.
\]

\end{proof}

Recall that an $R$-module $M\ $is said to be a \textit{torsion free} if the
set of torsion elements $T(M)=\{m\in M:rm=0\ $for some $0\neq r\in R\}$ of
$M\ $is zero. Also $M$ is called a torsion module if $T(M)=M.\ $We refer the
reader to \cite{AnChun} for more details on torsion subsets $T(M)\ $of $M.\ $

\begin{theorem}
\label{torsion} Every $S$-comultiplication module is either $S$-cyclic or torsion.
\end{theorem}

\begin{proof}
Let $M$ be an $S$-comultiplication module. Assume that $M$ is not an
$S$-cyclic module and $ann_{R}(m)=0$ for some $m\in M$. Since $Rm$ is a
submodule of $M$ and $M$ is an $S$-comultiplication module, we have
$s(0:_{M}ann(m))\subseteq Rm\subseteq(0:_{M}ann(m))$. It gives $sM\subseteq
Rm\subseteq M$ for some $s\in S$. This contradiction completes the proof.
Hence, $ann(m)\neq0$ for all $m\in M$ proving that $M$ is torsion module.
\end{proof}

\begin{theorem}
\label{tcy2}Let $R$ be an integral domain and $M$ be an $S$-finite and
$S$-comultiplication module. If $sM\ $is faithful for each $s\in S,\ $then $M$
is an $S$-cyclic module.
\end{theorem}

\begin{proof}
Suppose that $M$ is not an $S$-cyclic module. Then $M$ is a torsion module
from Theorem \ref{torsion}. Since $M$ is an $S$-finite module, there exist
$s\in S\ $and $m_{1},m_{2},\ldots,m_{n}\in M\ $such that $sM\subseteq
Rm_{1}+Rm_{2}+\cdots+Rm_{n}.$ This implies that $ann(Rm_{1}+Rm_{2}%
+\cdots+Rm_{n})=\bigcap\limits_{i=1}^{n}ann(m_{i})\subseteq ann(sM)=0\ $since
$sM\ $is faithful. Since $R\ $is an integral domain, there exists $m_{i}\in M$
such that $ann(m_{i})=0$ which is a contradiction. Hence, $M$ is an $S$-cyclic module.
\end{proof}

Recall from \cite{SenArTeKo} that an $R$-module $M\ $is said to be an
$S$-torsion free module if there exists $s\in S\ $and whenever $am=0\ $for
some $a\in R\ $and $m\in M,\ $then either $sa=0\ $or $sm=0.\ $

\begin{theorem}
\label{tcy3}Every $S$-comultiplication $S$-torsion free module is an
$S$-cyclic module.
\end{theorem}

\begin{proof}
Let $M$ be an $S$-comultiplication and $S$-torsion free module. If $sM=0$ for
some $s\in S,\ $then $M\ $is an $S$-cyclic module. So assume that $sM\neq
0\ $for each $s\in S.\ $Since $M\ $is an $S$-torsion free module, there exists
$t^{\prime}\in S\ $and whenever $am=0\ $for some $a\in R\ $and $m\in M,\ $then
either $t^{\prime}a=0\ $or $t^{\prime}m=0.\ $Since $t^{\prime}M\neq0,\ $there
exists $m\in M\ $such that $t^{\prime}m\neq0$.\ As $M\ $is an $S$%
-comultiplication module, there exists $t\in S\ $such that $t(0:_{M}%
ann(m))\subseteq Rm.\ $Since $ann(m)m=0$ and $M\ $is $S$-torsion free module,
we conclude either $t^{\prime}ann(m)=0\ $or $t^{\prime}m=0.\ $The second case
is impossible. So we have $t^{\prime}ann(m)=0\ $and so $t^{\prime}%
M\subseteq(0:_{M}ann(m)).\ $This implies that $t^{\prime}tM\subseteq
t(0:_{M}ann(m))\subseteq Rm$ where $t^{\prime}t\in S$, namely, $M\ $is an
$S$-cyclic module.
\end{proof}

Let $K\ $be a nonzero submodule of $M.\ K\ $is said to be an $S$-minimal
submodule if $L\subseteq K$\ for some submodule of $M,\ $then there exists
$s\in S\ $such that $sK\subseteq L.\ $

\begin{theorem}
Every $S$-comultiplication prime $R$-module $M\ $is $S$-minimal.
\end{theorem}

\begin{proof}
Let $M$ be an $S$-comultiplication prime $R$-module. Assume that $N$ is a
submodule of $M$. Since $M$ is prime, $ann(N)=ann(M)$. Also, $(0:_{M}%
ann(N))=(0:_{M}ann(M))$. Since $M$ is an $S$-comultiplication module,
$s(0:_{M}ann(N))\subseteq N\subseteq(0:_{M}ann(N))$ for some $s\in S$. Hence,
we get $s(0:_{M}ann(M))\subseteq N\subseteq(0:_{M}ann(M))$ and it shows that
$sM\subseteq N\subseteq M$. Therefore, $M$ is $S$-minimal.
\end{proof}

\section{S-second submodules of S-comultiplication modules}

This section is dedicated to the study of $S$-second submodules of
$S$-comultiplication module. Now, we need the following definition.

\begin{definition}
Let $M$ and $M^{\prime}\ $be two $R$-modules and $f:M\rightarrow M^{\prime}%
\ $be an $R$-homomorphism.

(i) If there exists $s\in S\ $such that $f(m)=0$ implies that $sm=0,\ $then
$f$ is said to be an $S$-injective (or, just $S$-monic).

(ii) If there exists $s\in S\ $such that $sM^{\prime}\subseteq
\operatorname{Im}f,\ $then $f$ is said to be an $S$-epimorphism (or, just $S$-epic).
\end{definition}

The following proposition is explicit. Let $M\ $be an $R$-module. An element
$x\in R\ $is called a zero divisor on $M\ $if there exists $0\neq m\in
M\ $such that $xm=0,\ $or equivalently,\ $ann_{M}(x)\neq(0).\ $The set of all
zero divisor elements of $R\ $on $M\ $is denoted by $z(M).\ $

\begin{proposition}
Let $M$ and $M^{\prime}\ $be two $R$-modules and $f:M\rightarrow M^{\prime}%
\ $be an $R$-homomorphism.

(i)\ $f$ is $S$-monic if and only if there exists $s\in S\ $such that
$sKer(f)=(0).$

(ii)\ If $f$ is monic, then $f$ is $S$-monic for each m.c.s $S\ $of $R.\ $The
converse holds in case $S\subseteq R-z(M).\ $

(iii)\ If $f$ is epic, then $f$ is $S$-epic for each m.c.s $S\ $of $R.\ $The
converse holds in case $S\subseteq u(R).\ $
\end{proposition}

Recall from \cite{SenArTeKo} that a submodule $P$ of $M\ $with $(P:M)\cap
S=\emptyset$ is said to be an $S$\textit{-prime submodule }if there exists a
fixed $s\in S\ $and whenever $am\in P\ $for some $a\in R,m\in M,\ $then either
$sa\in(P:M)$ or $sm\in P.\ $In particular, an ideal $I\ $of $R\ $is said to be
an $S$\textit{-prime ideal} if $I\ $is an $S$-prime submodule of $M$.\ We note
here that Acraf and Hamed, in their paper \cite{AcHa}, studied and
investigated the further properties of $S$-prime ideals. Now, we give the
following needed results which can be found in \cite{SenArTeKo}.

\begin{proposition}
\label{ppre}(i)\ (\cite[Proposition 2.9]{SenArTeKo}) If $P\ $is an $S$-prime
submodule of $M,\ $then $(P:M)\ $is an $S$-prime ideal of $R.$

(ii)\ (\cite[Lemma 2.16]{SenArTeKo}) If $P\ $is an $S$-prime submodule of
$M,\ $there exists a fixed $s\in S\ $such that $(P:_{M}s^{\prime}%
)\subseteq(P:_{M}s)\ $for each $s^{\prime}\in S.$

(iii)\ (\cite[Theorem 2.18]{SenArTeKo}) $P\ $is an $S$-prime submodule of
$M\ $if and only if $(P:_{M}s)\ $is a prime submodule of $M\ $for some $s\in
S.\ $
\end{proposition}

By the previous proposition, we deduce that $P\ $is an $S$-prime submodule if
and only if there exists a fixed $s\in S\ $such that $(P:_{M}s)\ $is a prime
submodule and $(P:_{M}s^{\prime})\subseteq(P:_{M}s)\ $for each $s^{\prime}\in
S.$

Sevim et al. in \cite{SenArTeKo} gave many characterizations of $S$-prime
submodules. Now, we give a new characterization of $S$-prime submodules from
another point of view.

Recall that a homomorphism $f:M\rightarrow M^{\prime}$ is said to be an
$S$\textit{-zero} if there exists $s\in S\ $such that $sf(m)=0\ $for each
$m\in M,\ $that is, $s\operatorname{Im}f=(0).\ $

\begin{proposition}
\label{pS-prime}Let $P$ be a submodule of $M\ $with $(P:M)\cap S=\emptyset
.\ $The following statements are equivalent.

(i)\ $P\ $is an $S$-prime submodule of $M.\ $

(ii)\ There exists a fixed $s\in S,\ $for any $a\in R\ $and$\ $the homothety
$M/P\overset{a.}{\rightarrow}M/P,\ $either $S$-zero or $S$-injective with
respect to $s\in S.\ $
\end{proposition}

\begin{proof}
$(i)\Rightarrow(ii):\ $Suppose that $P\ $is an $S$-prime submodule of
$M.\ $Then there exists a fixed $s\in S\ $such that $am\in P\ $for some $a\in
R,m\in M\ $implies that $saM\subseteq P\ $or $sm\in P.\ $Now, take $a\in
R\ $and assume that the homothety $M/P\overset{a.}{\rightarrow}M/P$ is not
$S$-injective with respect to $s\in S.\ $Then there exists $m\in M\ $such that
$a(m+P)=am+P=0_{M/P}\ $but $s(m+P)\neq0_{M/P}.\ $This gives that $am\in
P\ $and $sm\notin P.\ $Since $P$ is an $S$-prime submodule, we have
$sa\in(P:M)\ $and thus $sam^{\prime}\in P\ $for each $m^{\prime}\in M.\ $Then
we have $sa(m^{\prime}+P)=0_{M/P}\ $for each $m^{\prime}\in M,\ $that is, the
homothety $M/P\overset{a.}{\rightarrow}M/P$ is $S$-zero with respect to $s$.

$(ii)\Rightarrow(i):\ $Suppose that $(ii)$ holds. Let $am\in P\ $for some
$a\in R\ $and $m\in M.\ $Assume that $sm\notin P.\ $Then we deduce the
homothety $M/P\overset{a.}{\rightarrow}M/P$ is not $S$-injective. Thus by
(ii), $M/P\overset{a.}{\rightarrow}M/P$ is $S$-zero with respect to $s\in
S,\ $namely, $sa(m^{\prime}+P)=0_{M/P}\ $for each $m^{\prime}\in M.\ $This
yields that $sa\in(P:M).\ $Therefore, $P\ $is an $S$-prime submodule of $M.$
\end{proof}

It is well known that a submodule $P\ $of $M\ $is a prime submodule if and
only if every homothety $M/P\overset{a.}{\rightarrow}M/P$ is either injective
or zero. This fact can be obtained by Propositon \ref{pS-prime} by taking
$S\subseteq u(R).\ $

Recall from \cite{Fa} that a submodule $N\ $of $M\ $with $ann(N)\cap
S=\emptyset$ is said to be an $S$-second submodule if there exists $s\in
S,\ srN=0\ $or $srN=sN$ for each $r\in R.$\ Motivated by Proposition
\ref{pS-prime}, we give a new characterization of $S$-second submodules from
another point of view. Since the proof is similar to Proposition
\ref{pS-prime}, we omit the proof.

\begin{theorem}
\label{tsec}Let $N\ $be a submodule of $M\ $with $ann(N)\cap S=\emptyset
.\ $The following assertions are equivalent.

(i) $N\ $is an $S$-second submodule.

(ii)\ There exists $s\in S\ $such that for each $a\in R,\ $the
homothety$\ N\overset{a.}{\rightarrow}N\ $is either $S$-zero or $S$-surjective
with respect to $s\in S.\ $

(iii) There exists a fixed $s\in S$, for each $a\in R$, either $saN=0$ or
$sN\subseteq aN$.
\end{theorem}

The author in \cite{Fa} proved that if $N\ $is an $S$-second submodule of $M$,
then $ann(N)\ $is an $S$-prime ideal of $R\ $and the converse holds under the
assumption that $M\ $is comultiplication \cite[Proposition 2.9]{Fa}. Now, we
show that this fact is true even if $M\ $is an $S$-comultiplication module.

\begin{theorem}
\label{tm3}Let $M\ $be an $S$-comultiplication module. The following
statements are equivalent.

(i)\ $N\ $is an $S$-second sumodule of $M$.

(ii)\ $ann(N)\ $is an $S$-prime ideal of $R$ and there exists $s\in S\ $such
that $sN\subseteq s^{\prime}N$ for each $s^{\prime}\in S.\ $
\end{theorem}

\begin{proof}
$(i)\Rightarrow(ii):\ $The claim follows from \cite[Proposition 2.9]{Fa} and
\cite[Lemma 2.13]{Fa}.

$(ii)\Rightarrow(i):\ $Suppose that $ann(N)\ $is an $S$-prime ideal of
$R.\ $Now, we will show that $N\ $is an $S$-second submodule of $M.\ $To prove
this, take $a\in R.\ $Since $ann(N)\ $is an $S$-prime ideal, by Proposition
\ref{ppre}, there exists $s\in S\ $such that $ann(sN)$ is a prime ideal and
$ann(s^{\prime}N)\subseteq ann(sN)\ $for each $s^{\prime}\in S.\ $Assume that
$saN\neq(0).\ $Now, we shall show that $sN\subseteq aN.\ $Since $M\ $is an
$S$-comultiplication module, there exists $s^{\prime}\in S$ and an ideal
$I\ $of $R\ $such that $s^{\prime}(0:_{M}I)\subseteq aN\subseteq(0:_{M}%
I).\ $This implies that $aI\subseteq ann(N).\ $Since $ann(N)\ $is an $S$-prime
ideal, there exists $s\in S$ such that either$\ sa\in ann(N)\ $or $sI\subseteq
ann(N)\ $by Proposition \ref{ppre}. The first case impossible since
$saN\neq(0).\ $Thus we have $I\subseteq ann(sN).\ $Then we have $s^{\prime
}s(0:_{M}ann(sN))\subseteq s^{\prime}(0:_{M}I)\subseteq aN.\ $This implies
that $s^{\prime}s^{2}N\subseteq s^{\prime}s(0:_{M}ann(sN))\subseteq aN.\ $Then
by $(ii),\ sN\subseteq s^{\prime}s^{2}N\subseteq aN.\ $Then by Theorem
\ref{tsec} (iii), $N\ $is an $S$-second submodule of $M.\ $
\end{proof}

\begin{theorem}
Let $M\ $be a comultiplication module. The following statements are equivalent.

(i)\ $N$\ is a second submodule of $M.$

(ii)\ $ann(N)\ $is a prime ideal of $R.\ $
\end{theorem}

\begin{proof}
Take $S\subseteq u(R)$ and note that $S$-comultiplication module and
comultiplication modules are equal. On the other hand, second submodule and
$S$-second submodules are equivalent. The rest follows from Theorem \ref{tm3}.
\end{proof}

\begin{theorem}
Let $M\ $be an $S$-comultiplication module and let $N\ $be an $S$-second
submodule of $M.\ $If $N\subseteq N_{1}+N_{2}+\cdots+N_{m}\ $for some
submodules $N_{1},N_{2},\ldots,N_{m}\ $of $M,\ $then there exists $s\in
S\ $such that $sN\subseteq N_{i}\ $for some $1\leq i\leq m.$
\end{theorem}

\begin{proof}
Suppose that $N\ $is an $S$-second submodule of an $S$-comultiplication module
$M.\ $Suppose that $N\subseteq\sum\limits_{i=1}^{m}N_{i}$ for some submodules
$N_{1},N_{2},\ldots,N_{m}\ $of $M.\ $Then we have $ann(\sum\limits_{i=1}%
^{m}N_{i})=\bigcap\limits_{i=1}^{m}ann(N_{i})\subseteq ann(N).\ $Since $N\ $is
an $S$-second submodule, by Theorem \ref{tm3}, $ann(N)\ $is an $S$-prime ideal
of $R.\ $Then by \cite[Corollary 2.6]{SenArTeKo}, there exists $s\in S\ $such
that $sann(N_{i})\subseteq ann(N)\ $for some $1\leq i\leq m.\ $This implies
that $ann(N_{i})\subseteq ann(sN).\ $Then by Lemma \ref{lma} (iii),
$stN\subseteq N_{i}\ $for some $t\in S\ $which completes the proof.
\end{proof}


\begin{thebibliography}{99}                                                                                               %


\bibitem {AlSmi}Al-Shaniafi, Y., \& Smith, P. F. (2011). Comultiplication
modules over commutative rings. Journal of commutative algebra, 3(1), 1-29.

\bibitem {AnArTeKo}Anderson, D. D., Arabaci, T., Tekir, \"{U}., \& Ko\c{c}, S.
(2020). On S-multiplication modules. Communications in Algebra, 1-10.

\bibitem {AnChun}Anderson, D. D., \& Chun, S. (2014). The set of torsion
elements of a module. Communications in Algebra, 42(4), 1835-1843.

\bibitem {AnDu}Anderson, D. D., \& Dumitrescu, T. (2002). S-Noetherian rings.
Communications in Algebra, 30(9), 4407-4416.

\bibitem {AnFa1}Ansari-Toroghy, H., \& Farshadifar, F. (2012, December). On
the dual notion of prime submodules. In Algebra Colloquium (Vol. 19, No.
spec01, pp. 1109-1116). Academy of Mathematics and Systems Science, Chinese
Academy of Sciences, and Suzhou University.

\bibitem {AnFa2}Ansari-Toroghy, H., \& Farshadifar, F. (2012). On the dual
notion of prime submodules (II). Mediterranean journal of mathematics, 9(2), 327-336.

\bibitem {AnFa3}Ansari-Toroghy, H., \& Farshadifar, F. (2014, December). The
Zariski topology on the second spectrum of a module. In Algebra Colloquium
(Vol. 21, No. 04, pp. 671-688). Academy of Mathematics and Systems Science,
Chinese Academy of Sciences, and Suzhou University.

\bibitem {AnFa4}Ansari-Toroghy, H., \& Farshadifar, F. (2007). The dual notion
of multiplication modules. Taiwanese journal of mathematics, 1189-1201.

\bibitem {AnFa5}Ansari-Toroghy, H., \& Farshadifar, F. (2019). Some
generalizations of second submodules. Palestine journal of mathematics, 159-168.

\bibitem {HanKeyFar}Ansari-Toroghy, H., Keyvani, S., \& Farshadifar, F.
(2016). The Zariski topology on the Second spectrum of a module (II). Bulletin
of the Malaysian Mathematical Sciences Society, 39(3), 1089-1103.

\bibitem {At}Atani, R. E., \& Atani, S. E. (2009). Comultiplication modules
over a pullback of Dedekind domains. Czechoslovak mathematical journal, 59(4), 1103.

\bibitem {Bar}Barnard, A. (1981). Multiplication modules. J. ALGEBRA., 71(1), 174-178.

\bibitem {CeAlsmi}\c{C}eken, S., Alkan, M., \& Smith, P. F. (2013). Second
modules over noncommutative rings. Communications in Algebra, 41(1), 83-98.

\bibitem {CeAl2}\c{C}eken, S., \& Alkan, M. (2015). On the second spectrum and
the second classical Zariski topology of a module. Journal of Algebra and Its
Applications, 14(10), 1550150.

\bibitem {Ce}\c{C}eken, S. (2019). Comultiplication modules relative to a
hereditary torsion theory. Communications in Algebra, 47(10), 4283-4296.

\bibitem {Smi}El-Bast, Z. A., \& Smith, P. P. (1988). Multiplication modules.
Communications in Algebra, 16(4), 755-779.

\bibitem {Fa}F. Farshadifar, S-second submodules of a module, Algebra and
Discrete Mathematics, to appear.

\bibitem {AcHa}Hamed, A., \& Malek, A. (2019). S-prime ideals of a commutative
ring. Beitr\"{a}ge zur Algebra und Geometrie/Contributions to Algebra and
Geometry, 1-10.

\bibitem {SenArTeKo}Sevim, E. \c{S}., Arabaci, T., Tekir, \"{U}., \& Koc, S.
(2019). On S-prime submodules. Turkish Journal of Mathematics, 43(2), 1036-1046.

\bibitem {Yas}Yassemi, S. (2001). The dual notion of prime submodules. Arch.
Math.(Brno), 37(4), 273-278.
\end{thebibliography}
\end{document}